\documentclass[12pt]{amsart}
\usepackage[utf8]{inputenc}
\usepackage{amssymb}
\usepackage{amsmath}
\usepackage{comment}
\usepackage{amsthm}
\usepackage{textcomp}
\usepackage[all]{xy}
\usepackage{tikz-cd}
\usepackage{xcolor}
\definecolor{darkgrey}{rgb}{0.4,0.4,0.5}
\usepackage{hyperref}
\hypersetup{
	colorlinks=true,
	linkcolor=blue,
	citecolor=darkgrey
}

\usepackage[OT2,T1]{fontenc}
\DeclareSymbolFont{cyrletters}{OT2}{wncyr}{m}{n}
\DeclareMathSymbol{\Sha}{\mathalpha}{cyrletters}{"58}

\usepackage{url}

\usepackage{multirow}

\newtheorem{theorem}{Theorem}[section]

\newtheorem{conjecture}[theorem]{Conjecture}
\newtheorem{corollary}[theorem]{Corollary}
\newtheorem{proposition}[theorem]{Proposition}
\theoremstyle{definition}

\newtheorem{proposition-definition}[theorem]{Proposition-Definition}

\newtheorem{examples}[theorem]{Examples}

\theoremstyle{remark}
\newtheorem{remark}[theorem]{Remark}

\numberwithin{equation}{section}

\newcommand{\cyc}{{\mathrm{cyc}}}

\newcommand{\Gal}{{\mathrm{Gal}}}
\newcommand{\Ker}{{\mathrm{Ker}}}
\newcommand{\Coker}{{\mathrm{Coker}}}
\newcommand{\Sel}{{\mathrm{Sel}}}

\newcommand{\PG}{{\mathrm{PG}}}

\newcommand{\rank}{{\mathrm{rank}}}

\newcommand{\PGL}{{\mathrm{PGL}}}

\newcommand{\Q}{{\mathbb Q}}
\newcommand{\Z}{{\mathbb Z}}

\definecolor{Green}{rgb}{0.0, 0.5, 0.0}

\begin{document}
	\title[Iwasawa theory over $PGL(2)$ extension]{Selmer groups of elliptic curves over the $PGL(2)$ extension}

	\author{Jishnu Ray}
	\address{Department of Mathematics, The University of British Columbia}
	\curraddr{Room 121, 1984 Mathematics Road\\
		Vancouver, BC\\
		Canada V6T 1Z2}
	\email{jishnuray1992@gmail.com}
	
	\author{R. Sujatha}
	\address{Department of Mathematics, The University of British Columbia}
	\curraddr{Room 121, 1984 Mathematics Road\\
		Vancouver, BC\\
		Canada V6T 1Z2}
	\email{sujatha@math.ubc.ca}

	\thanks{~}

	\begin{abstract}
		Iwasawa theory of elliptic curves over noncommutative extensions has been a fruitful area of research. The central object of this paper is to use Iwasawa theory over the $GL(2)$ extension to study the dual Selmer group over the $PGL(2)$ extension. 
	\end{abstract}

	\maketitle
	
	\section{Introduction}
	Iwasawa theory of elliptic curves without complex multiplication over the trivializing extension was initiated in \cite{Howson_thesis}, \cite{CoatesHowsonII}, \cite{Coates_Fragments}. One of the important results proved in \cite{Coates_Fragments} is that the Pontryagin dual of the Selmer group is torsion over the corresponding noncommutative Iwasawa algebra under suitable hypothesis. 
	Let $E$ be an elliptic curve over a number field $F$ such that $E$ has good ordinary reduction at all the primes of $F$ lying over an odd prime $p$. A celebrated conjecture, due to Mazur,  then states that the dual Selmer group over the cyclotomic $\Z_p$-extension is torsion as a module over the corresponding Iwasawa algebra. 
	
	In the context of Iwasawa theory over   a noncommutative $p$-adic Lie extension, one 
	mainly considers  admissible $p$-adic Lie extensions which contain the cyclotomic extension.  Results from cyclotomic Iwasawa theory are then used to obtain results in noncommutative Iwasawa theory.
	Coates and the second author formulated the $\mathfrak{M}_H(G)$-conjecture as a natural analogue of the above conjecture of Mazur \cite{CoatesSujatha_MHG}. 
	In this paper, an attempt is made to study  Iwasawa theory  of  elliptic curves over a $PGL(2)$ extension, by adopting a descent approach, since the $PGL(2)$ extension does not contain the cyclotomic extension. Specifically, we use $GL(2)$ Iwasawa theory to try and gain insights into $PGL(2)$ Iwasawa theory.
	Our results, though modest give necessary and sufficient conditions for the dual Selmer group to be torsion over the $PGL(2)$ extension. In addition, we also use the structure theorem of modules over noncommutative Iwasawa algebras (see  \cite{CoatesSchneiderSujatha-modules}), to further understand when the dual Selmer group over the $PGL(2)$ extension is a torsion module. 
	Let $F_\infty=F[E_{p^\infty}]$ be the noncommutative $p$-adic Lie extension with Galois group $G=\Gal(F_\infty/F)$ having center $C$. Serre showed that  $G$ is a compact open subgroup of $GL(2,\Z_p)$. Let $\Sel(E/F_\infty)$ be the Selmer group over $F_\infty$; its  Pontryagin dual  is easily seen to be a finitely generated module over the Iwasawa algebra $\Lambda(G)$. Let $K_\infty$ be the fixed field of $F_\infty$ under the center $C$. Let  $\Sel(E/K_\infty)$ be the Selmer group over $K_\infty$;  its  dual  is  again a finitely generated module over the Iwasawa algebra $\Lambda(\PG)=\Lambda(G/C)$. The main result in the paper is summarized below.

		\begin{theorem}[see Theorem \ref{mainTheoremtorsion}]
			The dual Selmer group $\widehat{\Sel(E/K_\infty)}$ is a torsion  $\Lambda(\PG)$-module if any of the following conditions hold:
			\begin{enumerate}
				\item The  Selmer group $\Sel(E/F)$ is finite and $H_2(\PG, \widehat{\Sel(E/K_\infty)})$ is finite.
				\item The Selmer group $\Sel(E/F)$ is finite, $\widehat{\Sel(E/F_\infty)}$ is a torsion  $\Lambda(G)$-module and $H_0(\PG, H_1(C, \widehat{\Sel(E/F_\infty)})$ is finite.
				\item $\widehat{\Sel(E/F_\infty)}$ is a torsion $\Lambda(G)$-module and $C$ acts as a nonzero divisor on  $\widehat{\Sel(E/F_\infty)}$.
			\end{enumerate}
			Conversely, if  $\widehat{\Sel(E/K_\infty)}$ is a  torsion $\Lambda(\PG)$-module then  part $(3)$ holds.
		\end{theorem}
	
	The assertion in (3)  can also be  reinterpreted using the structure theorem for modules over noncommutative Iwasawa algebras. By the structure theorem    for finitely generated modules over noncommutative Iwasawa algebras  (see  \cite[p. 74]{CoatesSchneiderSujatha-modules}), the dual of $\Sel(E/F_\infty)$  is pseudo-isomorphic to a finite direct sum of cyclic modules of the form $\Lambda(G)/J_i$ where $J_i$ are reflexive  ideals in $\Lambda(G)$. If the center acts as a nonzero divisor on each of these cyclic summands, then it can be shown that the dual Selmer group $\Sel(E/K_\infty)$  is torsion as a $\Lambda(\PG)$-module (see Theorem \ref{thm:simplifiedMainTheoremTorsion}).   It seems to us that  the full strength of the structure theorem has not been exploited. Specifically, the ideals in the  summands are reflexive and pure of grade $1$. It might be possible to use this effectively to study the homology groups  $H_j(\PG, \Lambda(G)/J_i).$ This would then yield finer results on the structure of the dual Selmer group. We hope to return to this line of investigation  later.

	This paper consists of four sections including the introduction. In Section \ref{sec:prel},  we set up notation and collect the preliminary results that will be needed.
	In Section \ref{sec:3fundamental}, the defining exact sequences for the Selmer groups are used to compare the Selmer groups over the $GL(2)$ extension and the $PGL(2)$ extension (see Theorem \ref{thm:main}). Furthermore,  using a descent approach from the Selmer group over the $GL(2)$ extension, 
	we prove various equivalent conditions that are sufficient to ensure that the dual Selmer over the $PGL(2)$ extension is a cotorsion module over the corresponding Iwasawa algebra (see Theorems \ref{mainTheoremtorsion} and \ref{thm:simplifiedMainTheoremTorsion}). 
	As an application, Section \ref{Sec:applicationsall} deals with regular growth of ranks of Selmer groups as we descend from the $GL(2)$ extension to the $PGL(2)$ extension (see Proposition \ref{prop:reg_growth}).  Moreover,  we study the relation between the Selmer group over the base field $F$ and the Selmer group over the $PGL(2)$ extension (see Theorem \ref{thm:main2}).
	
	Zerbes \cite[Chapter 8]{Zerbes_thesis} and Howson \cite[Thm. 5.34]{Howson_thesis},  in their respective theses  studied the Euler characteristic of the Selmer group over the $PGL(2)$ extension, assuming it is a cotorsion module. Our attempt in this paper has largely been towards understanding when the dual Selmer over the $PGL(2)$ extension is a torsion module.
	
	It is clear that our methods in this paper are extendable to a broader class of Galois representations.  Our case of representations arising from elliptic curves should be seen as a first step.
	\section*{Acknowledgement}
	This work is supported by PIMS-CNRS  postdoctoral research funding received by the first author of this article.
	The authors would like to thank Srikanth B. Iyengar and Antonio Lei for helpful discussions. The first author would also like to thank the organizers of "Iwasawa 2019" conference (June 2019, Bordeaux) for invitation to speak. Theorem \ref{thm:main} in this paper was announced in that conference. The first author is  very grateful to Sarah Zerbes for sending him a copy of her PhD thesis on Selmer groups over $p$-adic Lie extensions.

	\section{Preliminaries}\label{sec:prel}
	
	Let $E$ be an elliptic curve over a number field $F$ without complex multiplication and $p$ be an odd prime. Let us suppose that $E$ has good ordinary reduction at  the primes above $p$. 
	
	For any $p$-adic Lie extension $H_\infty$ of $F$ we can define the Selmer group of $E$ over $H_\infty$ as a kernel of a natural global to local cohomological map defined by the following sequence.
	\begin{equation}\label{eq:def-Selmer}
	0 \rightarrow \Sel(E/H_\infty) \rightarrow H^1(F_S/H_\infty, E_{p^\infty}) \xrightarrow{\lambda_{H_\infty}} \oplus_{v \in S}J_v(H_\infty).
	\end{equation}
	Here $J_v(H_\infty)$'s are local cohomology groups defined in \cite[Eqn. 5-15]{Howson_thesis}.
	
	The Selmer group encodes several $p$-adic arithmetic information about the elliptic curve. It follows immediately from Kummer theory of $E$ over $H_\infty$ that we have the following exact sequence
	$$0 \rightarrow E(H_\infty) \otimes \Q_p/\Z_p \rightarrow \Sel(E/H_\infty) \rightarrow \Sha(E/H_\infty)(p) \rightarrow 0$$
	where $\Sha(E/H_\infty)$ is the  Tate-Shafarevich group (cf. \cite{CoatesSujatha_book}).

	It is easy to see that the dual Selmer group is a finitely generated module over the Iwasawa algebra 
	$$\Lambda(\Omega)= \varprojlim_W\Z_p[\Omega/W],$$ where $W$ runs over all open normal subgroups of $\Omega=\Gal(H_\infty/F)$. It is natural to study the structure of this Selmer group. In the classical case, when the group $\Omega=\Gal(F_\cyc/F)\cong \Z_p$ is commutative,  the Iwasawa algebra $\Lambda(\Omega)$ is a commutative local ring. One may then use the well known structure theorem of finitely generated modules in this setting \cite[Chap. VII, Sec. 4]{Bourbaki_commutative_algebra}. In the noncommutative case, the Iwasawa algebra $\Lambda(\Omega)$ is an Auslander regular local ring \cite[Thm. 3.26]{Venjakob1} and hence there is a dimension theory for modules in this setting. Suppose $M$ is a finitely generated module over a noncommutative Iwasawa algebra $\Lambda(\Omega)$. Then $M$ is a torsion $\Lambda(\Omega)$-module if  $\dim M \leq \dim \Lambda(\Omega) - 1 $. The module $M$ is pseudonull if $\dim M \leq \dim \Lambda(\Omega) - 2$ (see \cite[Sec. 3]{Venjakob1}). There is also a weak structure theorem \cite{CoatesSchneiderSujatha-modules}.

	Set $$F_n=F(E_{p^{n+1}}), \hspace{.5in} F_\infty = F(E_{p^{\infty}})$$
	and  write 
	$$G_n = \Gal(F_{\infty}/F_n), \hspace{.5in} G=\Gal(F_\infty/F).$$
	By a well known result of Serre \cite{Serre1}, $G$ is open in $GL_2(\Z_p)$ for all primes and $G=GL_2(\Z_p)$ for all but a finite number of primes. 
	Note that $F_\infty$ contains $F_\cyc$, the cyclotomic $\Z_p$-extension over $F$ with Galois group $\Gamma$, a $p$-adic Lie group of dimension $1$.

	A deep conjecture of \cite{Mazur} asserts
	\begin{conjecture}[Mazur]
		$\Sel(E/F_\cyc)$ is a  finitely generated  cotorsion  $\Lambda(\Gamma)$-module. 
	\end{conjecture}
	This conjecture is known to hold in
	some cases (for instance when $F=\Q$), thanks to a celebrated result of Kato \cite{Kato}.
	
	Coates and Howson in \cite{CoatesHowsonI}, \cite{CoatesHowsonII}, developed Iwasawa theory over the $GL(2)$ extension $F_\infty$ and proved conditions under which the Selmer group $\Sel(E/F_{\infty})$ is cotorsion as a module over the noncommutative Iwasawa algebra $\Lambda(G).$ In particular, they showed the following results (cf. Lemmas 4.7, 4.8, Prop. 4.3, Thm. 4.5 of \cite{Coates_Fragments}).

	\begin{theorem}[Coates, Howson]\label{thm:CoatesHowson}
		If $\Sel(E/F)$ is finite, then 
		$H^i(G,\Sel(E/F_\infty))$ is finite for $i=0,1$. Additionally, if $\Sel(E/F_\infty)$ is $\Lambda(G)$-cotorsion, then $H^i(G,\Sel(E/F_\infty))$ is  $0$ for $i =2,3,4$.
	\end{theorem}
	
	\begin{theorem}[Coates, Howson]\label{Coates:examples}
		If $G$ is a pro-$p$ group and $\Sel(E/F_{\cyc})$ is a cotorsion $\Lambda(\Gamma)$-module and has $\mu$-invariant $0$, then $\Sel(E/F_{\infty})$ is a cotorsion $\Lambda(G)$-module. 
	\end{theorem}
	
	Coates and Howson have also calculated an explicit formula for the Euler characteristic $\chi(G,\Sel(E/F_\infty)$ (cf. \cite[Thm. 1.1]{CoatesHowsonII}) under the assumption that $\Sel(E/F_\infty)$ is cotorsion as a $\Lambda(G)$-module. Their point of view was to understand how Iwasawa theory over the cyclotomic extension $F_\cyc$ influences the Iwasawa theory when one climbs up the tower to $F_\infty$.

	Let $M$ be a module with a continuous action of a  profinite group $G$ with a closed normal subgroup $H$. Recall the Hochschild-Serre spectral sequence,
	\begin{equation}\label{eq:hochserre}
	H^r(G/H,H^s(H,M)) \implies H^{r+s}(G,M).
	\end{equation}
	We shall make repeated use of this sequence with $H=C$, the centre of $G$.

	\section{Descent from $GL(2)$ Iwasawa theory to $PGL(2)$ Iwasawa theory}\label{sec:3fundamental}
	
	From the  Hochschild-Serre spectral sequence (see \eqref{eq:hochserre}) with $G=\Gal(F_\infty/F)$, $H=C$ and $G/H=\PG$, it is easy to see that we obtain the following commutative diagram with exact rows noting that the 
	the center $C$ is   pro-$p$ and has $p$-cohomological dimension 1.
	

	\begin{equation}\label{fundamental_diagram}
	\begin{tikzcd}
	0 \arrow[r] & \Sel(E/F_\infty)^C   \arrow[r] & H^1(F_S/F_\infty, E_{p^\infty})^C  \arrow[r] & \big(\oplus_{v \in S}J_v(F_\infty)\big)^C   \\
	0 \arrow[r] & \Sel(E/K_\infty) \arrow[u, "\alpha"] \arrow[r] &  H^1(F_S/K_\infty, E_{p^\infty}) \arrow[u, "\beta"]  \arrow[r, "\lambda_{K_\infty}"] & \oplus_{v \in S}J_v(K_\infty) \arrow[u, "\gamma"] 
	\end{tikzcd}
	\end{equation}
	
	The vertical maps are given by restriction maps. Our first objective is to prove the following theorem. 
	
	\begin{theorem}\label{thm:main}
		The vertical maps $\alpha, \beta, \gamma$ in the fundamental diagram \eqref{fundamental_diagram} are all isomorphisms. In particular, $$\Sel(E/K_\infty) \cong \Sel(E/F_\infty)^C$$ as $\Lambda(\PG)$-modules.
	\end{theorem}
	\begin{proof}
		Since the $p$-cohomological dimension of $C$ is $1$, by the Hochschild-Serre spectral sequence, it is clear that the cokernel of the maps $\beta$ and $\gamma$ are zero. The following argument shows that  $\ker (\beta)= H^1(C, E_{p^\infty})$ is zero.  The congruence subgroups of $GL_2(\Z_p)$ form a base of neighbourhood for its topology, thus $C$ must contain a scalar matrix
		\[
		x=
		\left( {\begin{array}{cc}
			1+p^n & 0 \\
			0 & 1+p^n \\
			\end{array} } \right)
		\]
		for $n$ sufficiently large. But $x$ lies in $C$, so $x-1$ annihilates  $H^1(C,E_{p^\infty})$ (see  \cite[Chap. I, Lemma 6.21]{Milne}).  Therefore, $p^nH^1(C,E_{p^\infty})=0.$ Now consider the short exact sequence 
		\begin{equation}\label{eq:Epn}
		0 \rightarrow E_{p^n} \rightarrow E_{p^\infty} \xrightarrow{\times p^n} E_{p^\infty} \rightarrow 0.
		\end{equation}
		As $p^nH^1(C,E_{p^\infty})=0$, the long exact sequence corresponding to \eqref{eq:Epn}  gives rise to the following short exact sequence
		\begin{equation}\label{a2}
		0 \rightarrow H^1(C, E_{p^\infty}) \rightarrow H^2(C,E_{p^n}) \rightarrow H^2(C,E_{p^\infty}) \rightarrow 0.
		\end{equation}
		Now, $C$ has $p$-cohomological dimension $1$ and hence $H^2(C,E_{p^n})=0$. Therefore $H^1(C, E_{p^\infty})=0$ by \eqref{a2}. This shows that the map $\beta$ is injective. 
		
		Writing  $\gamma=\oplus_{v \in S} \gamma_v$,  Shapiro's lemma (cf. \cite[Chap. VII, Prop. 7.2]{CasselFrohlich}, see also the proof of \cite[Lemma 2.8]{CoatesHowsonII})  gives
		\begin{equation}\label{eq:kernel-local map}
		\Ker(\gamma_v)=H^1(\Theta_{w}^C, E(F_{\infty, w}))(p),
		\end{equation}
		
		where $w$ is a prime of $F_\infty$ above $v$ and $\Theta_{w}^C$ is the decomposition group of $C$ at $w$. In the following, it will be shown  that
		\begin{equation}\label{impnote}
		H^1(\Theta_w^C,E_{p^\infty})=0.
		\end{equation}
		The decomposition subgroup $\Theta_w^C$ is a  subgroup of $C$ which is a subgroup of $\Z_p^\times = \mu_{p-1} \times (1+p\Z_p)$ where $\mu_{p-1}$ is the group of $(p-1)$-roots of unity. Therefore $\Theta_w^C$ is of the form $\Delta \times H$ where $\Delta$ is a finite group of order prime to $p$ and $H$ is a subgroup of $1+p\Z_p$ and hence a procyclic group. Let $\sigma$ be a topological generator of $H$, $\sigma=1+p^n$ for some $n$. By \cite[Prop. 1.7.7]{Neukirch}, $$H^1(H,E_{p^\infty})=E_{p^\infty}/(\sigma-1)E_{p^\infty}=E_{p^\infty}/p^nE_{p^\infty}.$$
		As $E_{p^\infty}$ is a $p$-divisible group, $E_{p^\infty}/p^nE_{p^\infty}=0$ , giving us $H^1(H,E_{p^\infty})=0$. 
		
		Also note that as $\Delta$ is a finite group of order prime to $p$, it follows that  $H^1(\Delta, E_{p^\infty}^H)=0.$
		The following  exact sequence
		$$0 \rightarrow H^1(\Delta,E_{p^\infty}^H) \rightarrow H^1(\Delta \times H, E_{p^\infty}) \rightarrow H^1(H,E_{p^\infty})^{\Delta},$$
		gives $H^1(\Delta \times H, E_{p^\infty})=0$ which proves \eqref{impnote}.
		
		If $v \nmid p$, by Kummer theory  \cite[Sec. 2]{Greenberg} $H^1(\Theta_{w}^C, E(F_{\infty, w}))(p)=H^1(\Theta_w^C,E_{p^\infty})$ which vanishes by \eqref{impnote}. This proves that $\Ker(\gamma_v)=0$ when $v \nmid p$.

		Now suppose that $v \mid p$.  Let $u$ be the prime of $K_\infty$ such that $w\mid u$ and $u\mid v$ and $I_u$ be the inertia subgroup of $K_\infty$ over $F$ at the prime $u$. Since $I_u$ is infinite, $K_{\infty, u}$ is deeply ramified (see  \cite{CoatesGreenberg}). Thus by \cite[Prop. 4.8, Thm. 2.13]{CoatesGreenberg},  $$\ker(\gamma_v) =H^1(\Theta_w^C,E(F_{\infty,w}))(p) \cong H^1(\Theta_w^C,D), $$
		where $D$ can be identified with $\Tilde{E}_{v,p^\infty}$, the $p$-primary subgroup of the reduction of $E$ modulo $v$.  By \cite{CoatesGreenberg}, the module $D$  also satisfies the  following exact sequence
		\begin{equation}\label{eq:D-seq}
		0 \rightarrow C^\prime \rightarrow E_{p^\infty} \rightarrow D \rightarrow 0
		\end{equation}
		
		where $C^\prime$ is divisible and $D$ is the maximal quotient of $E_{p^\infty}$ by a divisible subgroup such that $I_v$ acts on $D$ via a finite quotient. The exact sequence \eqref{eq:D-seq} gives that the following sequence $$ H^1(\Theta_w^C,E_{p^\infty}) \rightarrow H^1(\Theta_w^C,D) \rightarrow H^2(\Theta_w^C,C)$$ is exact. By  \eqref{impnote} we know that $H^1(\Theta_w^C,E_{p^\infty}) =0$ and since $\Theta_w^C$ is a subgroup of $C$ which has $p$-cohomological dimension $1$, we have $H^2(\Theta_w^C,C)=0$. Therefore, it is clear that $H^1(\Theta_w^C,D)=0 $ which shows that $\ker(\gamma_v)=0$ for $v \mid p$. 
		
		Thus the maps  $\beta$ and $\gamma$ is are isomorphisms.  The theorem now follows from the snake lemma applied to the fundamental diagram \eqref{fundamental_diagram}.

	\end{proof}
	
	\begin{corollary}\label{cor:descend}
		If $H^1(G, \Sel(E/F_\infty))$ is finite, then $H^1(\PG, \Sel(E/K_\infty))$ is also finite.
	\end{corollary}
	\begin{proof}
		The assertion follows from the natural  injection $$H^1\Big(PG, \big(\Sel(E/F_\infty)\big)^C\Big)\hookrightarrow H^1(G, \Sel(E/F_\infty)),$$
		and Theorem \ref{thm:main}.
	\end{proof}
	The reader is referred to Theorem \ref{thm:CoatesHowson}  in Section 2 for cases where 
	$H^1(G,\Sel(E/F_\infty))$ is known to be finite.
	\setcounter{subsection}{2}
	\subsection{Conditions when the Selmer over $\PGL(2)$ extension is cotorsion}\label{sec:main}
	\setcounter{theorem}{3}
	
	The following theorems give several conditions when   $\widehat{\Sel(E/K_\infty)}$ is torsion as a $\Lambda(\PG)$-module.
	\begin{theorem}
		Assume weak Leopoldt's conjecture at $K_\infty$, that is, $H^2(F_S/K_\infty, E_{p^\infty})=0$. Then the dual Selmer group $\widehat{\Sel(E/K_\infty)}$ is $\Lambda(\PG)$-torsion if and only if the map  $\lambda_{K_\infty}$ in \eqref{fundamental_diagram} is surjective.
	\end{theorem}
	\begin{proof}
		The proof follows from \cite[Thm. 7.2]{ShekharSujatha}. 
	\end{proof}
	
	\begin{theorem}\label{mainTheoremtorsion}
		The dual Selmer group $\widehat{\Sel(E/K_\infty)}$ is a torsion  $\Lambda(\PG)$-module if any of the following conditions hold:
		\begin{enumerate}
			\item The  Selmer group $\Sel(E/F)$ is finite and $H_2(\PG, \widehat{\Sel(E/K_\infty)})$ is finite.
			\item The Selmer group $\Sel(E/F)$ is finite, $\widehat{\Sel(E/F_\infty)}$ is a torsion  $\Lambda(G)$-module and $H_0(\PG, H_1(C, \widehat{\Sel(E/F_\infty)})$ is finite.
			\item $\widehat{\Sel(E/F_\infty)}$ is a torsion $\Lambda(G)$-module and $C$ acts as a nonzero divisor on  $\widehat{\Sel(E/F_\infty)}$.
		\end{enumerate}
		Conversely, if  $\widehat{\Sel(E/K_\infty)}$ is a  torsion $\Lambda(\PG)$-module then   $(3)$ holds.
	\end{theorem}

	\begin{proof}
		Let $M=\widehat{\Sel(E/F_\infty)}$. Since 
		$\Sel(E/F)$ is finite, the cohomology groups $H_0(\PG, M_C)$ and $H_1(\PG, M_C)$ are finite (see Theorem \ref{thm:CoatesHowson} and Corollary \ref{cor:descend}). Since
		\begin{align*}
		\rank_{\Lambda(\PG)}M_C &=\sum_{k \geq 0}^{3}(-1)^k\rank_{\Z_p}H_k(\PG,M_C),
		\end{align*}
		we obtain $\rank_{\Lambda(\PG)}M_C = 0$ and hence $M_C$ is $\Lambda(\PG)$-torsion. This proves (1).
		
		To show that the condition in (2) is also sufficient, note that if $M$ is $\Lambda(G)$-torsion and $\Sel(E/F)$ is finite,  then  $H_1(G,M)$ is finite and $H_2(G,M)=0$ (see Theorem \ref{thm:CoatesHowson}). By  Hochschild-Serre spectral sequence, we conclude that  $H_0(\PG, H_1(C, M))$ is finite if and only if $H_2(\PG,M)$ is finite, and hence (2) follows from (1).
		
		For (3), note that it follows from Hochschild-Serre spectral sequence that 
		\begin{align}
		\rank_{\Lambda(G)}M &=\sum_{k \geq 0}(-1)^k\rank_{\Z_p}H_k(G,M)\\
		&=\rank_{\Lambda(\PG)}M_C - \rank_{\Lambda(\PG)}H_1(C,M) \label{2}
		\end{align}
		
		Suppose (3) holds. Since $H_1(C,M)$  is precisely the $\Lambda(G)$-submodule of $M$ consisting of the elements in $M$ annihilated  by the augmentation ideal $I(C)=\langle c-1 \rangle$, we have $H_1(C,M)=0$. Hence the conclusion follows from \eqref{2} and Theorem \ref{thm:main}. 
		
		Finally, suppose
		if $M_C \cong \widehat{\Sel(E/K_\infty)}$ is a torsion $\Lambda(\PG)$-module. 
		Then it follows from \eqref{2} that $M$ is torsion as a $\Lambda(G)$-module and $H_1(C,M)$ is torsion as a $\Lambda(\PG)$-module. Then $H_1(C,M)$ is pseudonull as a $\Lambda(G)$-module. But $M$ has no nonzero pseudonull submodules (see \cite[Thm. 5.1]{OchiVenjakob_On_the_structure_of}) and therefore $H_1(C,M)=0$ and hence $C$ acts as a nonzero divisor on $M$.
		
	\end{proof}

	The following theorem  gives a restatement of condition (3) in Theorem \ref{mainTheoremtorsion} using the structure theorem of dual Selmer groups over noncommutative Iwasawa algebras \cite{CoatesSchneiderSujatha-modules}.
	By \textit{(loc.cit)}, there is an injection of $\Lambda(G)$-modules $$ \oplus_{i=1}^m\Lambda(G)/J_i \hookrightarrow M/M_0  $$ with pesudonull cokernel. Here $J_i$ 's are reflexive ideals in $\Lambda(G)$ which are pure of grade $1$ and $M_0$ is the maximal pseudonull submodule of $M$.   But $M$ has no nontrivial pseudonull submodule and therefore $M_0=0.$ This gives  the exact sequence 
	\begin{equation}\label{eq:exacttorsion}
	0 \rightarrow \oplus_{i=1}^m\Lambda(G)/J_i \rightarrow M \rightarrow N \rightarrow 0
	\end{equation}
	where $N$ is a pseudonull  $\Lambda(G)$-module. 
	
	\begin{theorem}\label{thm:simplifiedMainTheoremTorsion}
		The Selmer group $\Sel(E/K_\infty)$ is a cotorsion $\Lambda(\PG)$-module if and only if $\Sel(E/F_\infty)$ is a cotorsion $\Lambda(G)$-module and $C$ acts as a nonzero divisor on $\Lambda(G)/J_i$  for all $i \geq 1. $ 
	\end{theorem}
	\begin{proof}
		Since the center $C$ has $p$-cohomological dimension $1$, the sequence \eqref{eq:exacttorsion}, gives the following exact sequence of $\Lambda(\PG)$-modules.
		\begin{align}
		0\rightarrow H_1(C, \oplus_{i=1}^m \Lambda(G)/J_i) &\rightarrow H_1(C,M) \rightarrow H_1(C,N)  \label{exact1}\\
		&\rightarrow (\oplus_{i=1}^m\Lambda(G)/J_i)_C \rightarrow M_C \rightarrow N_C \rightarrow 0.\label{exact2}
		\end{align}

		If $C$ acts as a nonzero divisor on $\Lambda(G)/J_i$, then $J_i \nsubseteq I(C)$ and hence
		\begin{equation}\label{eq:short}
		\dim_{\Lambda(\PG)}(\Lambda(G)/J_i)_C <\dim\Lambda(\PG)=4,
		\end{equation}
		which implies that $(\Lambda(G)/J_i)_C $ is $\Lambda(\PG)$-torsion. Now, as $N$ is a pseudonull $\Lambda(G)$-module,    $$\dim_{\Lambda(G)}N \leq \dim\Lambda(G)-2=\dim\Lambda(\PG)-1.$$ This implies that $N_C$  is a torsion $\Lambda(\PG)$-module and the same holds for the module  $M_C$  from \eqref{eq:short} and \eqref{exact2}. 
		
		Conversely, suppose that $M_C \cong \widehat{\Sel(E/K_\infty)}$ is torsion as a $\Lambda(\PG)$-module. Then by Theorem \ref{mainTheoremtorsion}, $M$ is torsion as a $\Lambda(G)$-module and $C$ acts as a nonzero divisor on $M$. In particular, $H_1(C, M)=0$ whence $H_1(C, \oplus_{i=1}^m \Lambda(G)/J_i)=0$ by \eqref{exact1}. This implies that $C$ is a nonzero divisor on the summands $\Lambda(G)/J_i$.
		
	\end{proof}
	


	\begin{remark}
		We note that the ideals $J_i$ cannot be $I(C)$ for all $ i\geq 1$ because the center $C$ cannot act trivially on the Selmer group at $F_\infty$ (see Prop. \ref{prop:centerNonTrivial}).  
		
		Suppose none of the ideals $J_i$ are contained in $I(C)$.  Then  $\widehat{\Sel(E/K_\infty)}$ is torsion as a $\Lambda(\PG)$-module. On the other hand, if  $J_i=I(C)$ for some $i$, then  $\widehat{\Sel(E/K_\infty)}$ cannot be torsion as a $\Lambda(\PG)$-module (see Theorem \ref{thm:simplifiedMainTheoremTorsion}). 
		
		Suppose $\Sel(E/F_\infty)$ is  cotorsion as a $\Lambda(G)$-module, $\Sel(E/F)$ and $H_1(G,N)$ are finite.  Then none of the ideals $J_i$ can be contained in  $I(C)$.	 This is because, from  \eqref{eq:exacttorsion}  we obtain the exact sequence $$H_1(G,N) \rightarrow  (\oplus_{i=1}^m\Lambda(G)/J_i)_G \rightarrow M_G $$ whose first and last terms are finite. In this case, $\widehat{\Sel(E/K_\infty)}$ is torsion as a $\Lambda(\PG)$-module.
	\end{remark}

	In \cite[Prop. 8.10]{CoatesSchneiderSujatha-modules}, for the elliptic curve $E=X_1(11): y^2+y=x^3-x^2$ of conductor $11$ and prime $p=5$, the authors show that the center $C$ cannot act trivially on  $\Sel(E/F_\infty)$. The following proposition generalizes this for any elliptic curve without complex multiplication and with good ordinary reduction  for the primes above  $p$. 
	\begin{proposition}\label{prop:centerNonTrivial}
		Suppose $G \cong C \times \PG$, $\Sel(E/F_\infty)$ is a cotorsion $\Lambda(G)$-module and $\Sel(E/F)$ is finite. Then the center $C$ cannot act trivially on $\Sel(E/F_\infty)$.
	\end{proposition}
	\begin{proof}
		If $C$ acts trivially on $\Sel(E/F_\infty)$, then $$\Sel(E/F_\infty) =\Sel(E/F_\infty)^C \cong \Sel(E/K_\infty).$$
		As $\Sel(E/F_\infty)$ is $\Lambda(G)$-cotorsion and $\Sel(E/F)$ is finite, the cohomology groups  $H^i(G,\Sel(E/F_\infty))$ are finite for all $i$. The degeneration of the Hochschild-Serre spectral sequence (see \cite[Prop. 2.4.5]{Neukirch}), gives an injection $$H^i(\PG,\Sel(E/K_\infty)) \hookrightarrow H^i(G,\Sel(E/F_\infty)).$$
		This implies that the cohomology groups $H^i(\PG,\Sel(E/K_\infty))$ are finite for all $i$. Hence $$\rank_{\Lambda(\PG)}\widehat{\Sel(E/K_\infty)}=\sum_{i \geq 0}(-1)^i\rank_{\Z_p}H_i(\PG, \widehat{\Sel(E/K_\infty)})=0,$$
		whereby $\Sel(E/K_\infty) $ is $\Lambda(\PG)$-cotorsion. We deduce that $$\dim_{\Lambda(G)}\widehat{\Sel(E/F_\infty)} \leq \dim\Lambda(\PG) -1 = \dim\Lambda(G)-2.$$ This would imply that $\widehat{\Sel(E/F_\infty)}$ is a pseudonull $\Lambda(G)$-module. But  $\widehat{\Sel(E/F_\infty)}$ has no nonzero pseudonull submodules. Hence $\widehat{\Sel(E/F_\infty)}=0$.  On the other hand it is known that $\widehat{\Sel(E/F_\infty)}$ is infinite dimensional as a $\Q_p$-vector space (see \cite[Thm. 1.5]{Coates_Fragments}). This gives us a contradiction.
	\end{proof}
	\begin{examples}
		Here are some examples  of elliptic curves $E$ such that $G=\Gal(F_\infty/F)$ is a direct product of its center $C$ and  $\PG$. We follow the nomenclature from Cremona tables \cite{Cremona}.
		\begin{enumerate}
			\item Let  $E$ be the elliptic curve $X_1(11)$, namely $E$ is the curve $y^2+y=x^3-x^2$ and prime $p=5$. This is a curve of conductor $11$ defined over $\Q$ but we consider it over $F=\Q(\mu_5)$.  Put $F_\infty=F(E_{5^\infty})$. Then $G=\Gal(F_\infty/F)$ has the form $G=C \times PG$ (\cite[Example 8.7, p. 104]{CoatesSchneiderSujatha-modules}).
			
			\item Let $E$ be the elliptic curve $X_0(11)$, namely $E$ is the curve $y^2+y=x^3-x^2-10x-20$ and $p=5$. This is a curve of conductor $11$ defined over $\Q$ but we consider it over $F=\Q(\mu_5)$.  Put $F_\infty=F(E_{5^\infty})$. Then the Galois group $G=\Gal(F_\infty/F)$ is a subgroup of the first congruence kernel of $GL_2(\Z_5)$.  \cite[Eqn. 3, p. 586]{Fisher}. Hence $G$ is of the form $C \times \PG$.
			
			\item For any general elliptic curve $E$ over $F$ without complex multiplication and with good ordinary reduction  for the primes above  $p$, we can always find an integer $k$ large enough such that, over the base field $F[E_{p^k}]$, the Galois group $G=\Gal(F_\infty/F[E_{p^k}])$ lies inside the first congruence kernel of $GL_2(\Z_p)$ and hence can be written in the form $C \times \PG$. 
		\end{enumerate}
		
	\end{examples}
	\section{Applications}\label{Sec:applicationsall}
	Suppose $G \cong \PG \times C$. Let $C_n=C^{p^n}$ and $G_n=\PG \times C_n$. 
	\begin{proposition}\label{prop:reg_growth}
		Let $\widehat{\Sel(E/F_\infty)}$ be a torsion $\Lambda(G)$-module. Then, for all large $n$, $\rank_{\Lambda(\PG)}\widehat{\Sel(E/F_\infty)}_{C_n}$  is a constant, independent of  $n.$
	\end{proposition}
	\begin{proof}
		Let $M = \widehat{\Sel(E/F_\infty)}$. As $G_n$ is of finite index in $G$, $M$ is also a finitely generated module over $\Lambda(G_n)$. As $C_n \cong \Z_p$, we can identify $M^{C_n}$ with  $ H_1(C_n,M)$ and then they both are finitely generated over $\Lambda(\PG)$. Therefore, 
		\begin{align*}
		\rank_{\Lambda(\PG)}M_{C_n} &= \rank_{\Lambda(G_n)}M + \rank_{\Lambda(\PG)}M^{C_n}\\
		&=p^n\rank_{\Lambda(G)}M + \rank_{\Lambda(\PG)}M^{C_n}\\
		&=\rank_{\Lambda(\PG)}M^{C_n}
		\end{align*}
		Note that $C_n$ is in the center, hence abelian and therefore $M^{C_n}$ is a $\Lambda(G)$-submodule of $M$. 
		But $M$ is a finitely generated module over $\Lambda(G)$ and hence $M$ is a noetherian module and satisfies the ascending chain  condition on its submodules. Hence the chain $$M^{C_0=C} \subset M^{C_1} \subset \cdots M^{C_n} \cdots $$ stabilizes and so $\rank_{\Lambda(\PG)}M^{C_n}$ is a constant independent of $n$, for all sufficiently large $n$. 
	\end{proof}
	
	Consider the following fundamental diagram.

	\begin{equation}\label{fundamental_diagram2}
	\begin{tikzcd}
	0 \arrow[r] & \Sel(E/K_\infty)^\PG   \arrow[r] & H^1(F_S/K_\infty, E_{p^\infty})^\PG  \arrow[r, "\lambda_{K_\infty}^\PG"] & \big(\oplus_{v \in S}J_v(K_\infty)\big)^\PG   \\
	0 \arrow[r] & \Sel(E/F) \arrow[u, "f"] \arrow[r] &  H^1(F_S/F, E_{p^\infty}) \arrow[u, "g"]  \arrow[r, "\lambda_F"] & \oplus_{v \in S}H^1(F_v,E)(p) \arrow[u, "h"] 
	\end{tikzcd}
	\end{equation}
		Let $\Coker(\lambda_{K_\infty}^\PG)$ be the cokernel of the map $\lambda_{K_\infty}^\PG$ in \eqref{fundamental_diagram2}.
	\begin{theorem}\label{thm:main2}
		The vertical maps $f,g,h$ in the fundamental diagram \eqref{fundamental_diagram2} have finite kernels and cokernels. Furthermore, 	if $\Sel(E/F)$ is finite, then $\Coker(\lambda_{K_\infty}^\PG)$ is finite.
	\end{theorem}
	\begin{proof}
Since  $H^1(C, E_{p^\infty})=0$, using Hochschild-Serre spectral sequence  and noting that the cohomology groups $$H^i(\PG, E_{p^\infty}(K_\infty))=H^i(G,E_{p^\infty})$$ are finite for $i\geq 1$, it is easy to see that the kernel and the cokernel of $g$ are finite.  Arguing along the lines of Theorem \ref{thm:main}, using  Shapiro's lemma  it follows that the kernel and cokernel of $h$ are also finite. By snake lemma we deduce that the same is true for $f$. 

If $\Sel(E/F)$ is finite, $\Coker(\lambda_F)$ is finite (cf. \cite[p. 35]{CoatesSujatha_book}). This implies that $\Coker(h \circ \lambda_F)$ is finite. But $\Coker(h \circ \lambda_F)=\Coker(\lambda_{K_\infty}^\PG \circ g)$ and we know that $\Coker(g)$ is finite. Hence $\Coker(\lambda_{K_\infty}^\PG)$ is finite.

	\end{proof}

	\bibliographystyle{alpha}
	\bibliography{main}

\begin{thebibliography}{NSW08}

\bibitem[Bou98]{Bourbaki_commutative_algebra}
Nicolas Bourbaki.
\newblock {\em Commutative algebra. {C}hapters 1--7}.
\newblock Elements of Mathematics (Berlin). Springer-Verlag, Berlin, 1998.
\newblock Translated from the French, Reprint of the 1989 English translation.

\bibitem[CF67]{CasselFrohlich}
J.~W.~S. Cassels and A.~Fr\"{o}hlich.
\newblock {\em Algebraic Number Theory}.
\newblock Academic Press, 1967.

\bibitem[CG96]{CoatesGreenberg}
John Coates and Ralph Greenberg.
\newblock Kummer theory for abelian varieties over local fields.
\newblock {\em Invent. Math.}, 124(1-3):129--174, 1996.

\bibitem[CH97]{CoatesHowsonI}
John Coates and Susan Howson.
\newblock Euler characteristics and elliptic curves.
\newblock {\em Proc. Nat. Acad. Sci. U.S.A.}, 94(21):11115--11117, 1997.
\newblock Elliptic curves and modular forms (Washington, DC, 1996).

\bibitem[CH01]{CoatesHowsonII}
John Coates and Susan Howson.
\newblock Euler characteristics and elliptic curves. {II}.
\newblock {\em J. Math. Soc. Japan}, 53(1):175--235, 2001.

\bibitem[Coa99]{Coates_Fragments}
John Coates.
\newblock Fragments of the {${\rm GL}_2$} {I}wasawa theory of elliptic curves
  without complex multiplication.
\newblock In {\em Arithmetic theory of elliptic curves ({C}etraro, 1997)},
  volume 1716 of {\em Lecture Notes in Math.}, pages 1--50. Springer, Berlin,
  1999.

\bibitem[Cre97]{Cremona}
J.~E. Cremona.
\newblock {\em Algorithms for modular elliptic curves}.
\newblock Cambridge University Press, Cambridge, second edition, 1997.

\bibitem[CS10]{CoatesSujatha_book}
John Coates and Ramdorai Sujatha.
\newblock {\em Galois cohomology of elliptic curves}.
\newblock Published by Narosa Publishing House, New Delhi; for the Tata
  Institute of Fundamental Research, Mumbai, second edition, 2010.

\bibitem[CS12]{CoatesSujatha_MHG}
J.~Coates and R.~Sujatha.
\newblock On the {$\\mathfrak{M}_H(G)$}-conjecture.
\newblock In {\em Non-abelian fundamental groups and {I}wasawa theory}, volume
  393 of {\em London Math. Soc. Lecture Note Ser.}, pages 132--161. Cambridge
  Univ. Press, Cambridge, 2012.

\bibitem[CSS03]{CoatesSchneiderSujatha-modules}
John Coates, Peter Schneider, and Ramdorai Sujatha.
\newblock Modules over {I}wasawa algebras.
\newblock {\em J. Inst. Math. Jussieu}, 2(1):73--108, 2003.

\bibitem[Fis03]{Fisher}
Tom Fisher.
\newblock Descent calculations for the elliptic curves of conductor 11.
\newblock {\em Proc. London Math. Soc. (3)}, 86(3):583--606, 2003.

\bibitem[Gre99]{Greenberg}
Ralph Greenberg.
\newblock Iwasawa theory for elliptic curves.
\newblock In {\em Arithmetic theory of elliptic curves ({C}etraro, 1997)},
  volume 1716 of {\em Lecture Notes in Math.}, pages 51--144. Springer, Berlin,
  1999.

\bibitem[How98]{Howson_thesis}
Susan Howson.
\newblock {\em Iwasawa theory of Elliptic Curves for $p$-adic Lie Extensions}.
\newblock PhD thesis, University of Cambridge, 1998.

\bibitem[Kat04]{Kato}
Kazuya Kato.
\newblock {$p$}-adic {H}odge theory and values of zeta functions of modular
  forms.
\newblock {\em Ast\'{e}risque}, (295):ix, 117--290, 2004.
\newblock Cohomologies $p$-adiques et applications arithm\'{e}tiques. III.

\bibitem[Maz72]{Mazur}
Barry Mazur.
\newblock Rational points of abelian varieties with values in towers of number
  fields.
\newblock {\em Invent. Math.}, 18:183--266, 1972.

\bibitem[Mil86]{Milne}
J.~S. Milne.
\newblock {\em Arithmetic duality theorems}, volume~1 of {\em Perspectives in
  Mathematics}.
\newblock Academic Press, Inc., Boston, MA, 1986.

\bibitem[NSW08]{Neukirch}
J\"{u}rgen Neukirch, Alexander Schmidt, and Kay Wingberg.
\newblock {\em Cohomology of number fields}, volume 323 of {\em Grundlehren der
  Mathematischen Wissenschaften [Fundamental Principles of Mathematical
  Sciences]}.
\newblock Springer-Verlag, Berlin, second edition, 2008.

\bibitem[OV02]{OchiVenjakob_On_the_structure_of}
Yoshihiro Ochi and Otmar Venjakob.
\newblock On the structure of {S}elmer groups over {$p$}-adic {L}ie extensions.
\newblock {\em J. Algebraic Geom.}, 11(3):547--580, 2002.

\bibitem[Ser72]{Serre1}
Jean-Pierre Serre.
\newblock Propri\'{e}t\'{e}s galoisiennes des points d'ordre fini des courbes
  elliptiques.
\newblock {\em Invent. Math.}, 15(4):259--331, 1972.

\bibitem[SS12]{ShekharSujatha}
Sudhanshu Shekhar and R.~Sujatha.
\newblock On the structure of {S}elmer groups of {$\Lambda$}-adic deformations
  over {$p$}-adic {L}ie extensions.
\newblock {\em Doc. Math.}, 17:573--606, 2012.

\bibitem[Ven02]{Venjakob1}
Otmar Venjakob.
\newblock On the structure theory of the {I}wasawa algebra of a {$p$}-adic
  {L}ie group.
\newblock {\em J. Eur. Math. Soc. (JEMS)}, 4(3):271--311, 2002.

\bibitem[Zer05]{Zerbes_thesis}
Sarah Zerbes.
\newblock {\em Selmer groups over $p$-adic Lie extensions}.
\newblock PhD thesis, University of Cambridge, 2005.

\end{thebibliography}
\end{document}